\documentclass[10pt]{articlefederico2017}

\usepackage{amsthm}
\usepackage{amsmath}
\usepackage{amssymb}
\usepackage{color}
\usepackage{enumerate}
\usepackage{bbm}
\usepackage{datetime}
\usepackage{geometry}
\usepackage{tocloft}

\newtheorem{theorem}{Theorem}[section]
\newtheorem{corollary}[theorem]{Corollary}
\newtheorem{proposition}[theorem]{Proposition}
\newtheorem{lemma}[theorem]{Lemma}
\newtheorem{definition}[theorem]{Definition}
\newtheorem{defprop}[theorem]{Definition/Proposition}

\makeatletter
\newtheorem*{rep@theorem}{\rep@title} \newcommand{\newreptheorem}[2]{%
\newenvironment{rep#1}[1]{%
\def\rep@title{\bf #2 \ref{##1} }%
\begin{rep@theorem} }%
{\end{rep@theorem} } }
\makeatother
\newreptheorem{theorem}{Theorem}

\theoremstyle{definition}
\newtheorem{example}[theorem]{Example}

\newcommand{\conv}{\textrm{conv}}

\newcommand{\EVol}{\mathrm{EVol}}

\newcommand{\Vol}{\mathrm{Vol}}

\newcommand{\Span}{\mathrm{span}}
\newcommand{\children}{\mathrm{children}}
\newcommand{\rank}{\mathrm{rank}}   
   
\renewcommand{\Vert}{\mathrm{Vert}}
\newcommand{\id}{\mathrm{id}}
\newcommand{\inv}{\mathrm{inv}}

\newcommand{\N}{\mathbb{N}}
\newcommand{\Z}{\mathbb{Z}}

\newcommand{\R}{\mathbb{R}}

\newcommand{\w}{\overline{w}}

\usepackage{todonotes}

\begin{document}
\title{\textsf{The equivariant volumes of the permutahedron}}
\author{
\textsf{Federico Ardila\footnote{\noindent \textsf{San Francisco State University; Mathematical Sciences Research Institute; U. de Los Andes; federico@sfsu.edu.}}}
\and
\textsf{Anna Schindler\footnote{\noindent \textsf{San Francisco State University; North Seattle College; anna.schindler@seattlecolleges.edu.}}}
\and
\textsf{Andr\'es R. Vindas-Mel\'endez\footnote{\noindent \textsf{San Francisco State University; University of Kentucky; andres.vindas@uky.edu.
\newline 
The authors were supported by NSF Awards DMS-1855610 and DMS-1600609,  NSF Award DMS-1440140 to MSRI, and the Simons Foundation (FA), an ARCS Foundation Fellowship (AS), and NSF Graduate Research Fellowship DGE-1247392 
(ARVM).}}}}
\date{}
\maketitle

\begin{abstract} 
We prove that if $\sigma$ is a permutation of $S_n$ with $m$ cycles of lengths $l_1, \ldots, l_m$, the subset of the permutahedron $\Pi_n$ fixed by the natural action of $\sigma$ is a polytope with volume $n^{m-2} \gcd(l_1, \ldots, l_m)$.
\end{abstract}

\section{Introduction}

The $n$-permutahedron is the polytope in $\R^n$
whose vertices are the $n!$ permutations of $[n]:=\{1, \ldots, n\}$:
\[
\Pi_n := \conv \left\{(\pi(1), \pi(2), \dots,\pi(n)): \pi \in S_n\right\}.
\]
The symmetric group $S_n$ acts on $\Pi_n \subset \R^n$ by permuting coordinates; more precisely, a permutation $\sigma \in S_n$ acts on a point $x = (x_1, x_2, \dots, x_n) \in \Pi_n$, by $\sigma \cdot x := (x_{\sigma^{-1}(1)}, x_{\sigma^{-1}(2)}, \dots, x_{\sigma^{-1}(n)})$.

\begin{definition}
The \emph{fixed polytope} or \emph{slice  of the permutahedron $\Pi_n$ fixed by a permutation $\sigma$ of $[n]$} is
\[
\Pi_n^\sigma = \{x \in \Pi_n \, : \, \sigma \cdot x = x\}.
\]
\end{definition}


\begin{figure}[h]
    \centering
    \includegraphics[scale=.54]{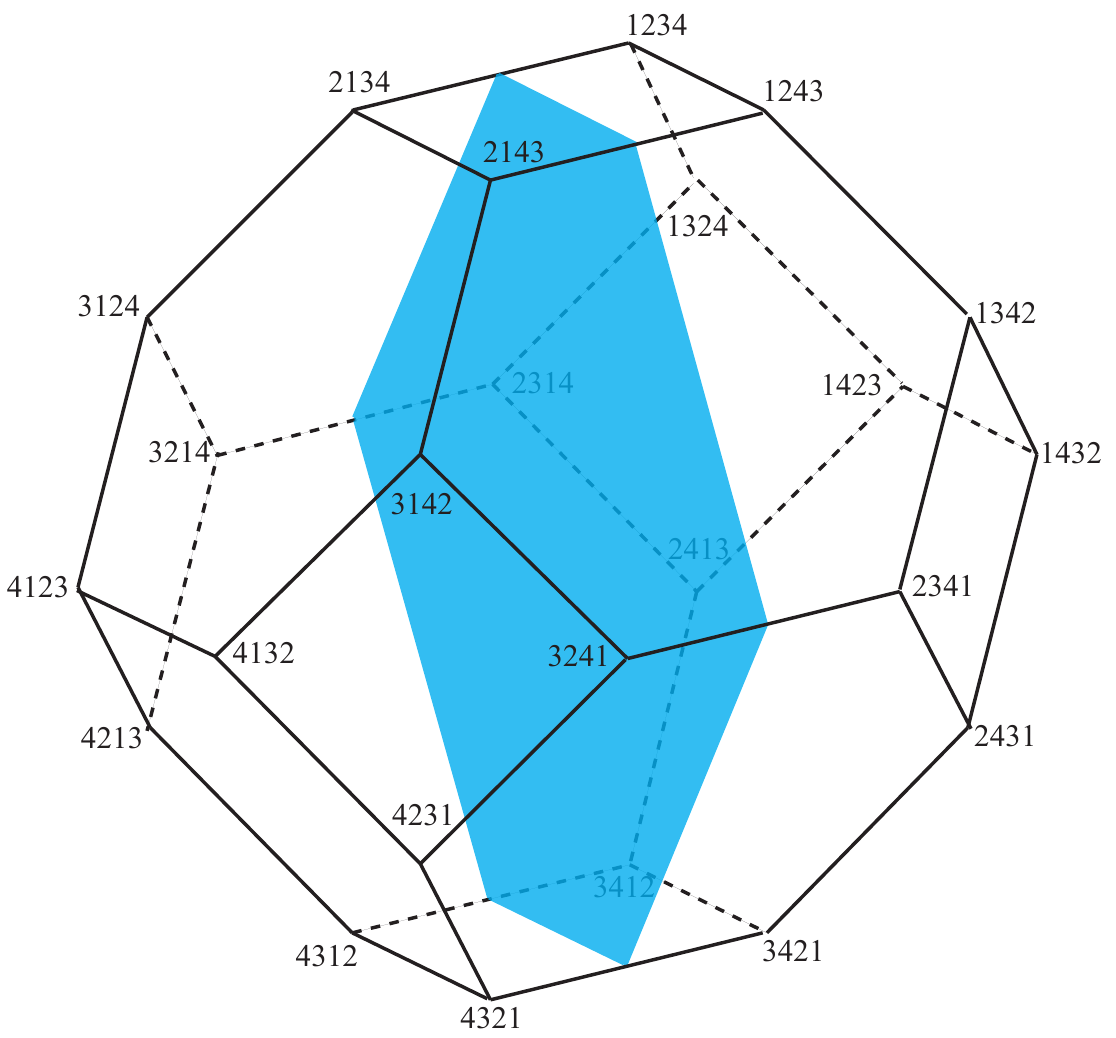} 
    \caption{The slice $\Pi_4^{(12)}$ of the permutahedron $\Pi_4$ fixed by $(12) \in S_4$ is a hexagon.}
    \label{fig:fixedby(12)}
\end{figure}

%

The main result of this note is a generalization of the fact, due to Stanley \cite{Stanleyzonotope}, that $\Vol\, \Pi_n$ equals $n^{n-2}$, the number of spanning trees on $[n]$.

\begin{theorem} \label{thm:main}
If $\sigma$ is a permutation of $[n]$ whose cycles have lengths $l_1, \ldots, l_m$, then the normalized volume of the slice of $\Pi_n$ fixed by $\sigma$ is 
\[
\Vol\,  \Pi_n^\sigma = n^{m-2} \gcd(l_1, \ldots, l_m).
\]
\end{theorem}

This is the first step towards describing the equivariant Ehrhart theory of the permutahedron, a question posed by Stapledon \cite{Stapledon}.

\subsection{Normalizing the volume}

The permutahedron and its fixed polytopes are not full-dimensional; we must define their volumes carefully. We normalize volumes so that every primitive parallelotope has volume 1. This is the normalization under which the volume of $\Pi_n$ equals $n^{n-2}$. 

More precisely, let $P$ be a $d$-dimensional polytope on an affine $d$-plane $L \subset \Z^n$. Assume $L$ is integral, in the sense that $L \cap \Z^n$ is a lattice translate of a $d$-dimensional lattice $\Lambda$. We call a lattice $d$-parallelotope in $L$ \emph{primitive} if its edges generate the lattice $\Lambda$; all primitive parallelotopes have the same volume. Then we define the volume of a $d$-polytope $P$ in $L$ to be $\Vol(P) := \EVol(P)/\EVol(\square)$ for any primitive parallelotope $\square$ in $L$, where $\EVol$ denotes Euclidean volume. By convention, the normalized volume of a point is $1$.

The definition of $\Vol(P)$  makes sense even when $P$ is not an integral polytope. This is important for us because the fixed polytopes of the permutahedron are 
not necessarily integral.

\subsection{Notation}

We identify each permutation $\pi \in S_n$ with the point $(\pi(1), \ldots, \pi(n))$ in $\R^n$. When we write permutations in cycle notation, we do not use commas to separate the entries of each cycle. For example, we identify the permutation $246513$ in $S_6$ with the point $(2,4,6,5,1,3) \in \R^6$, and write it as $(1245)(36)$ in cycle notation.

Our main object of study is the fixed polytope $\Pi_n^\sigma$ for a permutation $\sigma \in S_n$. We assume that $\sigma$ has $m$ cycles $\sigma_1, \ldots, \sigma_m$ of lengths $l_1 \geq \cdots \geq l_m$. 

We let $\{e_1, \ldots, e_n\}$ be the standard basis of $\R^n$, and $e_S := e_{s_1} + \cdots + e_{s_k}$ for $S = \{s_1, \ldots, s_k\} \subseteq [n]$. Recall that the Minkowski sum of polytopes $P, \, Q \subset \R^n$ is the polytope $P+Q := \{p + q \, : \, p \in P, q \in Q\} \subset \R^n$. \cite{Grunbaum}

\subsection{Organization}

Section \ref{sec:threedescriptions} is devoted to proving Theorem \ref{thm:big}, which describes the fixed polytopes $\Pi_n^\sigma$ in terms of its vertices, its defining inequalities, and a Minkowski sum decomposition. Section \ref{sec:volume} uses this to prove our main result, Theorem \ref{thm:main}, that the normalized volume of $\Pi_n^\sigma$ is $n^{m-2} \gcd(l_1, \ldots, l_m)$. Section \ref{sec:closing} contains some closing remarks. These include a connection with Reiner's theory of equivariant subdivisions of polytopes, and a discussion of the slice of the permutahedron fixed by a subgroup of $S_n$.

\section{Three descriptions of the fixed polytopes of the permutahedron} \label{sec:threedescriptions}


\begin{proposition}\cite{Ziegler} \label{thm:perm}
The permutahedron $\Pi_n$ can be described in the following three ways:

\begin{enumerate}

\item (Inequalities)
It is the set of points $x \in \R^n$ satisfying
\begin{enumerate}
\item $x_1 + x_2 + \cdots + x_n = 1 + 2 + \cdots + n$, and
\item $x_{i_1} + x_{i_2} + \cdots + x_{i_k} \geq 1 + 2 + \cdots + k$ 
for any subset $\{i_1, i_2, \dots, i_k\} \subset \{1,2, \dots, n\}$.
\end{enumerate}

\item (Vertices)
It is the convex hull of the points $(\pi(1), \ldots, \pi(n))$ as $\pi$ ranges over the permutations of $[n]$.

\item (Minkowski sum)
It is the Minkowski sum $\displaystyle \sum_{1 \leq j < k \leq n} [e_k, e_j] + \sum_{1 \leq k \leq n} e_k.$
\end{enumerate}
The $n$-permutahedron is $(n-1)$-dimensional and every permutation of $[n]$  is indeed a vertex.
\end{proposition}

Our first goal is to prove the analogous result for the fixed polytopes of $\Pi_n$; we do so in Theorem \ref{thm:big}.

\subsection{Standardizing the permutation}

We define the \emph{cycle type} of a permutation $\sigma$ to be the partition of $n$ consisting of the lengths $l_1 \geq \cdots \geq l_m$ of the cycles $\sigma_1, \ldots, \sigma_m$ of $\sigma$.

\begin{lemma}\label{lemma:conjugation}
The volume of $\Pi_n^\sigma$ only depends on the cycle type  of $\sigma$. 
\end{lemma}

\begin{proof}
Two permutations of $S_n$ have the same cycle type if and only if they are conjugate \cite{Sagan}. For any two conjugate permutations $\sigma$ and $\tau\sigma\tau^{-1}$ (where $\sigma, \tau \in S_n$) we have
\begin{equation}\label{eq:conjugate}
\Pi_n^{\tau\sigma\tau^{-1}} = \tau \cdot \Pi_n^\sigma.
\end{equation}
Every permutation $\tau \in S_n$ acts isometrically on $\R^n$ because $S_n$ is generated by the transpositions $(i \,\,\,\, i+1)$ for $1 \leq i \leq n-1$, which act as reflections across the hyperplanes $x_i = x_{i+1}$. It follows from \eqref{eq:conjugate} that the fixed polytopes $\Pi_n^{\tau\sigma\tau^{-1}}$ and $\Pi_n^\sigma$ have the same volume, as desired.
\end{proof}

We wish to understand the various fixed polytopes of $\Pi_n$, and \eqref{eq:conjugate} shows that we can focus our attention on the slices $\Pi_n^\sigma$ fixed by a permutation of the form 
\begin{equation}\label{eq:sigma}
\sigma=(1 \,\,\, 2\,\,\,  \ldots\,\,\,  l_1)(l_1+1 \,\,\, l_1+2 \,\,\, \ldots \,\,\,   l_1+l_2) \cdots (l_1+\cdots+l_{m-1}+1 \,\,\,  \ldots \,\,\, n-1 \,\,\,  n)
\end{equation}
for a partition $l_1 \geq l_2 \geq \cdots \geq l_m$ with $l_1+\cdots+l_m = n$. 
We do so from now on.

\subsection{The inequality description}

\begin{lemma} \label{lemma:eqs}
For a permutation $\sigma \in S_n$, the fixed polytope $\Pi_n^\sigma$ consists of the points $x \in \Pi_n$ satisfying $x_j = x_k$ for any $j$ and $k$ in the same cycle of $\sigma$.
\end{lemma}

\begin{proof}
Suppose that $x \in \Pi_n^{\sigma}$. First, let $h$ and $i$ be adjacent entries in a cycle $\sigma_a$ of $\sigma$, with $\sigma(h) = i$. Since $\sigma \cdot x = x$, we have
\[
x_i = (\sigma \cdot x)_i = x_{\sigma^{-1}(i)} = x_h.
\]
This holds for any adjacent entries of $\sigma_a$, so by transitivity $x_j=x_k$ for  any two entries $j,k$ of $\sigma_a$.

Conversely, suppose $x \in \Pi_n$ is such that $x_j=x_k$ whenever $j$ and $k$ are in the same cycle of $\sigma$. For any $1 \leq i \leq n$, let $h$ be the index preceding $i$ in the appropriate cycle of $\sigma$, so $\sigma(h)=i$. Then we have that $(\sigma \cdot x)_i = x_{\sigma^{-1}(i)} = x_h = x_i$. Since this holds for any index $i$, we have $\sigma \cdot x = x$ as desired.
\end{proof}

Geometrically, Lemma \ref{lemma:eqs} tells us that the fixed polytope $\Pi_n^\sigma$ is the slice of $\Pi_n$ cut out by the hyperplanes $x_j=x_k$ for all pairs $j,k$ such that $j$ and $k$ are in the same cycle of $\sigma$. 
For example, the fixed polytope $\Pi_4^{(12)}$ is the intersection of $\Pi_4$ with the hyperplane $x_1=x_2$, as shown in Figure \ref{fig:fixedby(12)}.

\begin{corollary}
If a permutation $\sigma$ of $[n]$ has $m$ cycles then $\Pi_n^{\sigma}$ has dimension $m-1$. 
\end{corollary}

\begin{proof}
Let $\sigma=\sigma_1 \cdots \sigma_m$ be the cycle decomposition of $\sigma$. A cycle $\sigma_j = (a_1\ a_2\ \cdots\ a_{l_j})$ of length $l_j$ imposes $l_j-1$ linear conditions on a point $x$ in the fixed polytope, namely $x_{a_1} = x_{a_2} = \cdots = x_{a_{l_j}}$. Because $\sigma$ has $m$ cycles whose lengths add up to $n$, we have a total of $n-m$ such conditions, and they are linearly independent. The fixed polytope $\Pi_n^{\sigma}$ is the transversal intersection of $\Pi_n$ with these $n-m$ linearly independent hyperplanes, so $\dim \Pi_n^{\sigma} = \dim \Pi_n - (n-m) = m-1$.
\end{proof}

\subsection{Towards a vertex description} \label{sec:vertices}

In this section we describe a set $\Vert(\sigma)$ of $m!$ points associated to a permutation $\sigma$ of $S_n$. We will show in Theorem \ref{thm:big} that this is the set of vertices of the fixed polytope $\Pi_n^\sigma$. 

\begin{definition}\label{def:v}
The set $\Vert(\sigma)$ of $\sigma$-vertices consists of the $m!$ points 
\[
\overline{v_{\prec}} := \sum_{k=1}^m \bigg( \dfrac{l_k +1}{2}
+ \sum_{j \, : \, \sigma_j \prec \sigma_k} l_j  \bigg) e_{\sigma_k}
\]
as $\prec$ ranges over the $m!$ possible linear orderings of $\sigma_1, \sigma_2, \dots, \sigma_m$.
\end{definition}

These points can be nicely described in terms of the map taking a point to the average of its $\sigma$-orbit.

\begin{defprop}\label{def:average}
The \emph{$\sigma$-average map} is the projection $\overline{\,\cdot\,} :  \R^n \rightarrow \R^n$ taking a point $w$ to the average $\overline{w}$ of its $\sigma$-orbit. If $|\sigma|$ is the order of $\sigma$ as an element of the symmetric group $S_n$, we have:
\begin{align}\label{eq:orbit}
\w \,\,:=& \,\,\frac1{|\sigma|} \sum_{i=1}^{|\sigma|} \sigma^i \cdot w \nonumber \\
=&\,\, \sum_{k=1}^m\dfrac{\sum_{j \in \sigma_k} w_j}{l_k} \, e_{\sigma_k}.
\end{align}
\end{defprop}

\begin{proof}
To prove the equality of these two expressions, let $w_{\sigma_k}$ denote the projection of $w$ to the coordinates in $\sigma_k$, so the $i$th coordinate of $w_{\sigma_k}$ equals $w_i$ if $i \in \sigma_k$ and $0$ otherwise. Thus, $w = w_{\sigma_1} + \cdots + w_{\sigma_m}$ and 
\[
    \overline{w} = 
\sum_{k=1}^m \overline{w_{\sigma_k}} 
= \sum_{k=1}^m \dfrac{1}{|\sigma|} \sum_{i=1}^{|\sigma|} \sigma^i w_{\sigma_k}\\
    =\dfrac{1}{|\sigma|} \sum_{k=1}^m \dfrac{|\sigma|}{|\sigma_k|} \sum_{i=1}^{|\sigma_k|} \sigma_k^i w_{\sigma_k},
\]
because $\sigma_k$ is the only cycle that acts on $w_{\sigma_k}$ non-trivially, and $|\sigma|$ is a multiple of $|\sigma_k|$. For each cycle $\sigma_k$ we have 
\[
\sum_{i=1}^{|\sigma_k|}\sigma_k^i w_{\sigma_k} = \Big(\sum_{j \in \sigma_k} w_j\Big) e_{\sigma_k},
\]
from which the desired equiality follows.
\end{proof}

\begin{proposition}\label{prop:sigmaperm}
Given $\sigma \in S_n$, we say a permutation $v = (v_1, \dots, v_n)$ of $[n]$ is  \emph{$\sigma$-standard} if it satisfies the following property: for each cycle $(j_1\ j_2\ \cdots \ j_r)$ of $\sigma$, $(v_{j_1}, v_{j_2}, \dots, v_{j_r})$ is a sequence of consecutive integers in increasing order. The set of \emph{$\sigma$-vertices} equals
\[
\Vert(\sigma) := \{\w \, : \, w \textrm{ is a $\sigma$-standard permutation of } [n]\}
\]
with no repretitions.
\end{proposition}

\begin{proof}
If $v= (v_1, \dots, v_n)$ is a $\sigma$-standard permutation, then for each cycle $(a_1\ a_2\ \dots\  a_r)$ of $\sigma$, $v_{a_1},\dots, v_{a_r}$ is an increasing sequence of consecutive integers. The placement of these integers determines a linear ordering $\prec$ of $\sigma_1, \dots, \sigma_m$ as follows.
The smallest cycle in $\prec$ is the cycle $\sigma_a$ whose coordinates are the integers $1, 2, \dots, l_a$, the second smallest is the cycle $\sigma_b$ whose coordinates are the integers $l_a +1, l_a + 2, \dots, l_a + l_b$, and so on.
Any linear order $\prec$ of the cycles corresponds to a unique $\sigma$-standard permutation $v_{\prec}$ in this way.

Now, we can use \eqref{eq:orbit} to compute $\overline{v_{\prec}}$:
for each cycle $\sigma_k$,  the set $\{v_i \, : \, i \in \sigma_k\}$ consists of the integers from 
$ 1+ \hspace{-.2cm}\sum\limits_{j \, : \, \sigma_j \prec \sigma_k} l_j $ to  $l_k+\hspace{-.2cm} \sum\limits_{j \, : \, \sigma_j \prec \sigma_k} l_j $, whose average is $ \frac12(l_k+1) + \hspace{-.2cm} \sum\limits_{j \, : \, \sigma_j \prec \sigma_k} l_j$, in agreement with Definition \ref{def:v}. 
\end{proof}

\begin{example}
For $\sigma=(1234)(567)(89)$, the $\sigma$-standard permutations in $S_9$ are 
\begin{align*}
(1,2,3,4,\,\, 5,6,7, \,\, 8,9), \qquad &  (1,2,3,4, \,\, 7,8,9, \,\, 5,6), \\
(4,5,6,7, \,\, 1,2,3, \,\, 8,9),\qquad  &  (3,4,5,6, \,\, 7,8,9, \,\, 1,2),  \\
(6,7,8,9, \,\, 1,2,3, \,\, 4,5), \qquad  & (6,7,8,9, \,\, 3,4,5, \,\, 1,2),
\end{align*}
and the corresponding $\sigma$-vertices are
\begin{align*}
&\tfrac{1+2+3+4}{4}\, e_{1234} +\tfrac{5+6+7}{3}\,e_{567} +\tfrac{8+9}{2}\,e_{89}, \quad
\tfrac{1+2+3+4}{4}\,e_{1234} +\tfrac{7+8+9}{3}\,e_{567} +\tfrac{5+6}{2}\,e_{89},\\
&\tfrac{4+5+6+7}{4}\,e_{1234} +\tfrac{1+2+3}{3}\,e_{567} +\tfrac{8+9}{2}\,e_{89}, \quad
\tfrac{3+4+5+6}{4}\,e_{1234} +\tfrac{7+8+9}{3}\,e_{567} +\tfrac{1+2}{2}\,e_{89}, \\
&\tfrac{6+7+8+9}{4}\,e_{1234} +\tfrac{1+2+3}{3}\,e_{567} +\tfrac{4+5}{2}\,e_{89}, \quad
\tfrac{6+7+8+9}{4}\,e_{1234} +\tfrac{3+4+5}{3}\,e_{567} +\tfrac{1+2}{2}\,e_{89}.
\end{align*}
\end{example}

The following standard observation will be very important.

\begin{lemma} \label{lem:averaging}
The image of the permutahedron $\Pi_n$ under the $\sigma$-averaging map $\overline{\,\cdot\,}$ is the fixed polytope $\Pi_n^\sigma$.
\end{lemma}

\begin{proof} 
For any point $w \in \Pi_n$, the average $\overline{w}$ of its $\sigma$-orbit is in $\Pi_n$ and is $\sigma$-fixed, so it is in $\Pi_n^\sigma$. Conversely, any point $p \in \Pi_n^\sigma$ satisfies $p = \overline{p}$, and hence is in the image of $\overline{\,\cdot\,}$.
\end{proof}

\subsection{Towards a zonotope description}\label{sec:zonotope}

We will show in Theorem \ref{thm:big} that the fixed polytope $\Pi_n^\sigma$ is the following zonotope.

\begin{definition}\label{def: M sigma}
Let $M_{\sigma}$ denote the Minkowski sum
\begin{align} \label{Msigma}
M_{\sigma}\,\, := \,\, & \sum_{1 \leq j < k \leq m}[l_je_{\sigma_k}, l_ke_{\sigma_j}]+ \sum_{k=1}^m \dfrac{l_k+1}{2}e_{\sigma_k} \nonumber \\
= \,\,  & \sum_{1 \leq j < k \leq m}[0, l_ke_{\sigma_j} - l_je_{\sigma_k}]+ \sum_{k=1}^m \bigg(\dfrac{l_k+1}{2}  + \sum 
_{j<k} l_j \bigg)e_{\sigma_k}.
\end{align}
\end{definition}

Two polytopes $P$ and $Q$ are \emph{combinatorially equivalent} if their  posets of faces, partially ordered by inclusion, are isomorphic. 
They are \emph{linearly equivalent} if there is a bijective linear function mapping $P$ to $Q$. They are \emph{normally equivalent} if they live in the same ambient vector space and have the same normal fan. 

\begin{proposition}\label{prop: comb eq}
The zonotope $M_\sigma$ is combinatorially equivalent to the standard permutahedron $\Pi_m$, where $m$ is the number of cycles of $\sigma$.
\end{proposition}

 \begin{proof}
The $\sigma$-fixed subspace of $\R^n$ is $ (\mathbb{R}^n)^{\sigma} := \{x \in \R^n \, : \, x_j = x_k \textrm{ if $j$ and $k$ are in the same cycle of $\sigma$}\}$, so 
 $\{e_{\sigma_1}, \ldots, e_{\sigma_m}\}$ is a basis for it.  
 Let $\{f_1, \ldots, f_m\}$ be the standard basis for $\R^m$ and define the linear bijective map $\phi : \mathbb{R}^m \rightarrow (\mathbb{R}^n)^{\sigma}$ by 
\begin{equation} \label{eq:phi}
\phi(f_i) = \frac1{l_i} e_{\sigma_i} \textrm{ for } 1 \leq i \leq m.
\end{equation}
This map shows that $\sum_{j<k} [f_k, f_j]$ is linearly equivalent to $\sum_{j<k} \left[e_{\sigma_k} /l_k\, ,\,  e_{\sigma_j}/l_j\right].$

The normal fan of a Minkowski sum $P_1 + \cdots P_s$ is the coarsest common refinement of the normal fans of $P_1$, \ldots, $P_s$  \cite[Prop. 7.12]{Ziegler}. Therefore, scaling each summand does not change the normal fan of a Minkowski sum. Thus
$\sum_{j<k} \left[e_{\sigma_k} /l_k\, ,\,  e_{\sigma_j}/l_j\right]$
is normally equivalent to 
$\sum_{j<k} l_jl_k \left[e_{\sigma_k} /l_k\, ,\,  e_{\sigma_j}/l_j\right] = $
$\sum_{j<k} \left[l_je_{\sigma_k} \, ,\,  l_ke_{\sigma_j} \right].$

%
Finally, since $\sum_{j<k} [f_k, f_j]$ and $\sum_{j<k} \left[l_j e_{\sigma_k}, l_k e_{\sigma_j}\right]$ are translates of $\Pi_m$ and $M_\sigma$, respectively, the desired result follows from the fact that linear and normal equivalence implies combinatorial equivalence.
 \end{proof}

\subsection{The three descriptions of the fixed polytope are correct}

Now we are ready to prove that the vertex and Minkowski sum descriptions of the fixed polytope presented in Sections \ref{sec:vertices} and \ref{sec:zonotope} are correct.

\begin{theorem} \label{thm:big}
Let $\sigma$ be a permutation of $[n]$ whose cycles $\sigma_1, \ldots, \sigma_m$ have respective lengths $l_1, \ldots, l_m$. The fixed polytope $\Pi_n^\sigma$ can be described in the following four ways:

\begin{enumerate}
\item[0.]
It is the set of points $x$ in the permutahedron $\Pi_n$ such that $\sigma \cdot x = x$.
\item
It is the set of points $x \in \R^n$ satisfying
\begin{enumerate}
\item $x_1 + x_2 + \cdots + x_n = 1 + 2 + \cdots +  n$,
\item $x_{i_1} + x_{i_2} + \cdots + x_{i_k} \geq 1 + 2 + \cdots + k$ for any subset $\{i_1, i_2, \dots, i_k\} \subset \{1,2, \dots, n\}$, \textrm{ and} 
\item $x_i=x_j$ for any $i$ and $j$ which are in the same cycle of $\sigma$.
\end{enumerate}
\item
It is the convex hull of the set $\Vert(\sigma)$ of $\sigma$-vertices described in Definition \ref{def:v}.
\item
It is the Minkowski sum $M_{\sigma} = \displaystyle \sum_{1 \leq j < k \leq m}[l_je_{\sigma_k}, l_ke_{\sigma_j}]+ \sum_{k=1}^m \dfrac{l_k+1}{2}e_{\sigma_k}$
of Definition \ref{def: M sigma}.
\end{enumerate}
Consequently, the fixed polytope $\Pi_n^\sigma$ is a zonotope that is combinatorially isomorphic to the permutahedron $\Pi_m$. It is $(m-1)$-dimensional and every $\sigma$-vertex is indeed a vertex of $\Pi_n^\sigma$.
\end{theorem}

\begin{proof}
Description 0. is the definition of the fixed polytope $\Pi_n^\sigma$, and we already proved in Lemma \ref{lemma:eqs} that description 1. is correct. It remains to prove that
\[
\Pi_n^{\sigma} = \conv(\Vert(\sigma)) = M_{\sigma} .
\]
We proceed in three steps as follows:
\[
A. \quad \conv(\Vert(\sigma)) \subseteq \Pi_n^{\sigma} \qquad \qquad 
B. \quad M_{\sigma} \subseteq \conv(\Vert(\sigma)) \qquad \qquad 
C. \quad \Pi_n^{\sigma} \subseteq M_{\sigma}
\]

\noindent $A. \qquad   \conv(\Vert(\sigma)) \subseteq \Pi_n^{\sigma}: $
This follows directly from Proposition \ref{prop:sigmaperm} and Lemma \ref{lem:averaging}.
%

\medskip

\noindent $B. \qquad   M_{\sigma} \subseteq \conv(\Vert(\sigma)): $ 
It suffices to show that any vertex $v$ of $M_{\sigma}$ is in $\Vert(\sigma)$.
Consider a linear functional $c= (c_1, c_2, \dots, c_n) \in (\mathbb{R}^n)^*$ such that $v = (M_{\sigma})_c$ is the face of $M_\sigma$ where $c\in (\R^n)^*$ is maximized. For $k = 1,\ldots, m$, let
\[
c_{\sigma_k} := \frac1{l_k} \sum_{i \in \sigma_k} c_i.
\]

First we claim that $c_{\sigma_j} \neq c_{\sigma_k}$ for $j \neq k$.
Minkowski sums satisfy $(P+Q)_c = P_c + Q_c$ \cite[Eq. 2.4]{Sturmfelspolytopes}, so
\begin{equation}\label{eq:maxface}
v = (M_{\sigma})_c = \sum_{j < k}[l_je_{\sigma_k}, l_ke_{\sigma_j}]_c + \sum_{k=1}^{m} \frac{l_k +1}{2}e_{\sigma_k}.
\end{equation}
Thus each summand $[l_je_{\sigma_k}, l_ke_{\sigma_j}]_c$ must be a single point, equal to either $l_je_{\sigma_k}$ or $l_ke_{\sigma_j}$. Therefore,
\[
c(l_j e_{\sigma_k}) = l_j \sum_{i \in \sigma_k} c_i = l_jl_k c_{\sigma_k} 
\qquad \textrm{ and } \qquad 
c(l_k e_{\sigma_j}) = l_k \sum_{i \in \sigma_j} c_i = l_jl_k c_{\sigma_j} 
\]
are distinct, hence $c_{\sigma_j} \neq c_{\sigma_k}$, as desired. We also see that
$
[l_je_{\sigma_k}, l_ke_{\sigma_j}]_c$
equals
$
l_je_{\sigma_k}
$ if 
$c_{\sigma_j} < c_{\sigma_k},$ and it equals 
$l_ke_{\sigma_j}$ if
$c_{\sigma_j} > c_{\sigma_k}$.


%
Now that we know that $c_{\sigma_1}, c_{\sigma_2}, \dots, c_{\sigma_m}$ are strictly ordered, we let $\prec$ be the corresponding linear order on $\sigma_1, \sigma_2, \dots, \sigma_m$. We then have that
\[
v =  \sum_{k=1}^m \Big(  \sum_{j \, : \, c_{\sigma_j}< c_{\sigma_k}} l_j \Big) e_{\sigma_k} + \sum_{k=1}^m \frac{l_k+1}{2} e_{\sigma_k}
    = \overline{v_{\prec}} \in \Vert(\sigma),
\]  
as desired.

%

\bigskip

\noindent $C. \qquad  \Pi_n^{\sigma} \subseteq M_{\sigma}:$
By Lemma \ref{lem:averaging},
it suffices to show that $\overline{\tau} \in M_{\sigma}$ for all permutations $\tau$. To do so, let us first derive an alternative expression for $\overline{\tau}$. We begin with the vertex $\id = (1,2, \dots, n)$ of $\Pi_n$. The identity permutation is $\sigma$-standard, so
\begin{equation}\label{eq:id}
\overline{\id} = \sum_{k=1}^m \bigg( \dfrac{l_k +1}{2}\ + \sum_{j<k} l_j \bigg) e_{\sigma_k}.
\end{equation}
Notice that this is the translation vector for the Minkowski sum of  \eqref{Msigma}.

Now, let us compute $\overline{\tau}$ for any permutation $\tau$. Let 
\[
l = \inv(\tau) =  |\{(a,b) \, : \, 1 \leq a < b \leq n, \,\, \tau(a) > \tau(b)\}|
\]
be the number of inversions of $\tau$. Consider a minimal sequence $\id = \tau_0, \tau_1, \ldots, \tau_l = \tau$ of permutations such that $\tau_{i+1}$ is obtained from $\tau_i$ by exchanging the positions of numbers $p$ and $p+1$, thus introducing a single new inversion without affecting any existing inversions. Such a sequence corresponds to a minimal factorization of $\tau$ as a product of simple transpositions $(p \,\,\, p+1)$ for $1 \leq p \leq n-1$. We have $\inv(\tau_i) = i$ for $1 \leq i \leq l$. 

Now we compute $\overline{\tau}$ by analyzing how $\overline{\tau_i}$ changes as we introduce new inversions, using that
\begin{equation}\label{eq:contributions}
\overline{\tau} - \overline{\id} = 
(\overline{\tau_{l}} - \overline{\tau_{l-1}}) + \cdots + 
(\overline{\tau_{1}} - \overline{\tau_{0}}).
\end{equation}
If $a<b$ are the positions of the numbers $p$ and $p+1$ that we switch as we go from $\tau_i$ to $\tau_{i+1}$, 
then
regarding $\tau_i$ and $\tau_{i+1}$ as vectors in $\R^n$ we have
\[
\tau_{i+1} - \tau_i = e_a-e_b.
\]
If $\sigma_j$ and $\sigma_k$ are the cycles of $\sigma$ containing $a$ and $b$, respectively, we have 
\begin{equation} \label{eq:costofaswap}
\overline{\tau_{i+1}} - \overline{\tau_i} =  \overline{e_a} -\overline{e_b} 
=  \frac{e_{\sigma_j}}{l_j} - \frac{e_{\sigma_k}}{l_k} = \frac1{l_jl_k} (l_k e_{\sigma_j} - l_j e_{\sigma_k})
\end{equation}
in light of \eqref{eq:orbit}. 
This is the local contribution to \eqref{eq:contributions} that we obtain when we introduce a new inversion between a position $a$ in cycle $\sigma_j$ and a position $b$ in cycle $\sigma_k$ in our permutation. Notice that this contribution is $0$ when $j=k$. Also notice that we will still have an inversion between positions $a$ and $b$ in all subsequent permutations, due to the minimality of the sequence. We conclude that
\begin{equation}\label{eq:idtotau}
  \overline{\tau} - \overline{\id} =  \sum_{j<k} \dfrac{\inv_{j,k}(\tau)}{l_jl_k}(l_k e_{\sigma_j} - l_j e_{\sigma_k})
\end{equation}
where 
\[
\inv_{j,k}(\tau) =  |\{(a,b) \, : \, 1 \leq a < b \leq n,  \,\, a \in \sigma_j, \,\, b \in \sigma_k\,\, \textrm{ and } \tau(a) > \tau(b)\}|
\]
is the number of inversions in $\tau$ between a position in $\sigma_j$ and a position in $\sigma_k$ for $j<k$.

Equations \eqref{eq:id} and \eqref{eq:idtotau} give us an alternative description for $\overline{\tau}$. This description makes it apparent that $\overline{\tau} \in M_\sigma$: Notice that $|\sigma_j|=l_j$ and $|\sigma_k|=l_k$ imply that $0 \leq \inv_{j,k}(\tau) \leq l_jl_k$, so
\[
  \overline{\tau} - \overline{\id} \in \sum_{1 \leq j < k \leq n} [0, l_k e_{\sigma_j} - l_j e_{\sigma_k}];
\]  
combining this with \eqref{Msigma} and \eqref{eq:id} gives the desired result.
\end{proof}

 \begin{figure}[h]
    \centering   
     \includegraphics[height=3.9cm]{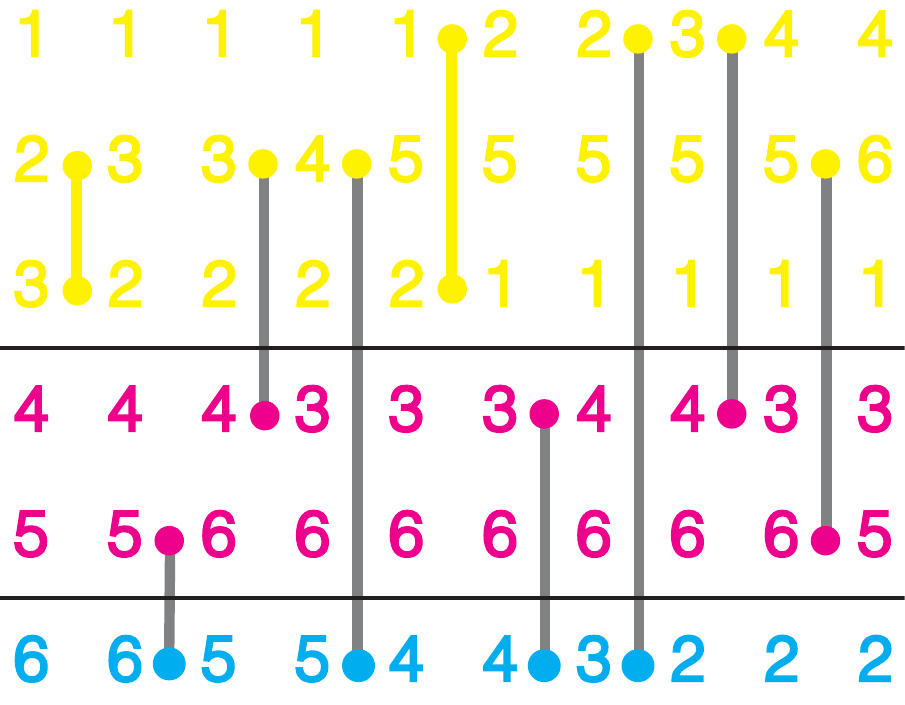} \qquad \qquad \qquad
    \includegraphics[height=3.4cm]{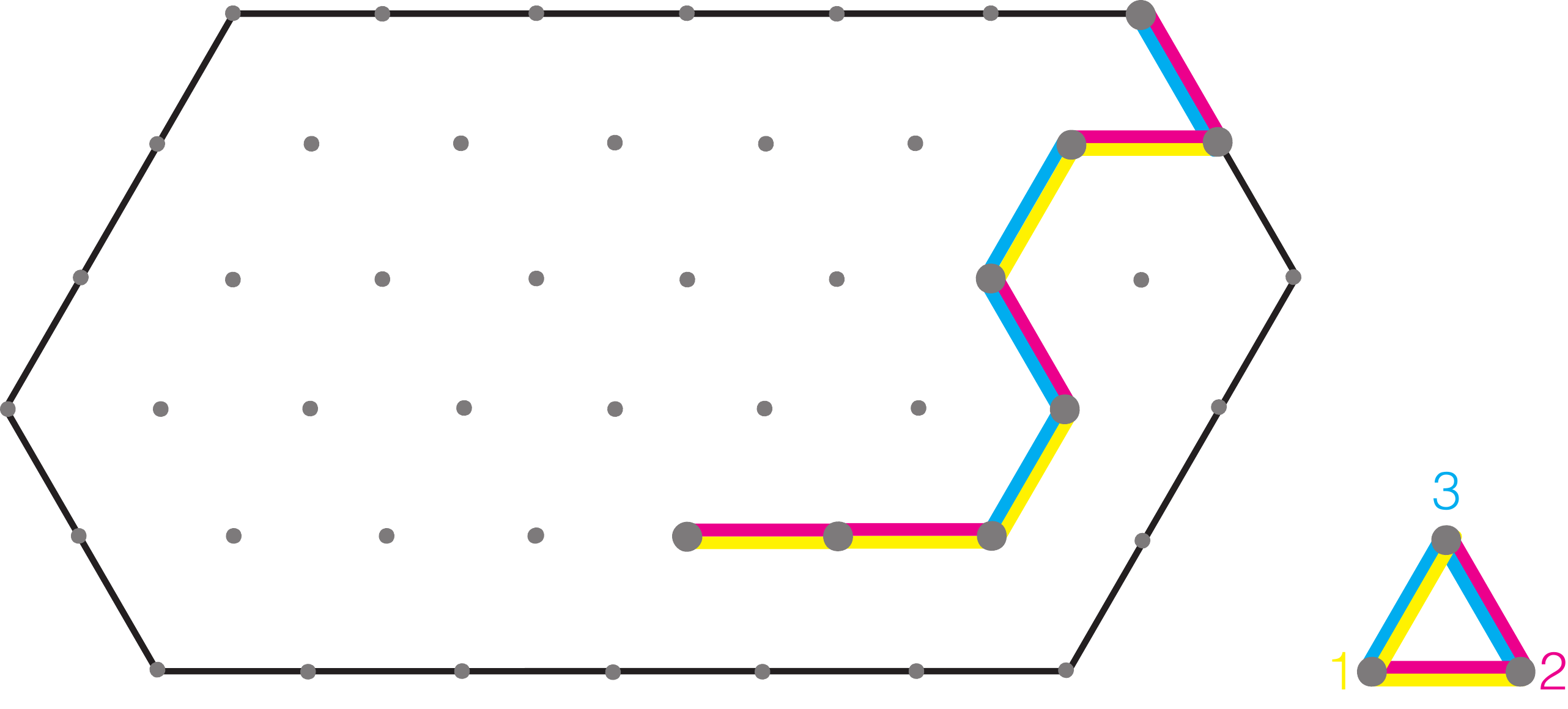}
    \caption{(a) A minimal sequence of permutations  $\id=\tau_0, \tau_1, \ldots, \tau_9 = 461352$ adding one inversion at a time and (b) the corresponding path from $\overline{\id}$ to $\overline{\tau}$ in the zonotope $M_\sigma$.}
    \label{fig:pathinzonotope}
\end{figure}

\begin{example}
Figure \ref{fig:pathinzonotope} illustrates part C of the proof above for $n=6$, $\sigma=(123)(45)(6)$, and the permutation $\tau = 461352$. This permutation has $\inv(\tau)=9$ inversions, and the columns of the left panel show a minimal sequence of permutations $\id = \tau_0, \tau_1, \ldots, \tau_9 = \tau$ where each $\tau_{i+1}$ is obtained from $\tau_i$ by swapping two consecutive numbers, thus introducing a single new inversion. 

The rows of the diagram are split into three groups 1, 2, and 3, corresponding to the support of the cycles of $\sigma$. Out of the $\inv(\tau)=9$ inversions of $\tau$, there are $\inv_{1,2}(\tau) = 3$ involving groups $1$ and $2$, $\inv_{1,3}(\tau) =2$ involve groups $1$ and $3$, and $\inv_{2,3}(\tau) =2$ involving groups $2$ and $3$.

This sequence of permutations gives rise to a walk from $\overline{\id}$, which is the top right vertex of the zonotope $M_\sigma$, to $\overline{\tau}$. In the rightmost triangle, which is not drawn to scale, vertex $i$ represents the point $e_{\sigma_i}/l_i$ for $1 \leq i \leq 3$. Whenever two numbers in groups $j<k$ are swapped in the left panel, to get from permutation $\tau_i$ to $\tau_{i+1}$, we take a step in direction $e_{\sigma_j}/l_j - e_{\sigma_k}/l_k$ in the right panel, to get from point $\overline{\tau_i}$ to $\overline{\tau_{i+1}}$. This is the direction of edge $jk$ in the triangle, and its length is $1/l_jl_k$ of the length of the generator $l_ke_{\sigma_j} - l_je_{\sigma_k}$ of the zonotope.
Then
\[
\overline{\tau} - \overline{\id} =
\frac{3}{l_1l_2} (l_2 e_{\sigma_1} - l_1 e_{\sigma_2}) +
\frac{2}{l_1l_3} (l_3 e_{\sigma_1} - l_1 e_{\sigma_3}) +
\frac{2}{l_2l_3} (l_3 e_{\sigma_2} - l_2 e_{\sigma_3}).
\]
Since $3 =  \inv_{1,2}(\tau) \leq l_1l_2 = 6$,\,  $2 = \inv_{1,3}(\tau) \leq l_1l_3 = 3$ and $2 =\inv_{2,3}(\tau) \leq l_2l_3 = 2$, the resulting point $\overline{\tau}$ is in the zonotope $M_\sigma$.
\end{example}

\section{The volumes of the fixed polytopes of $\Pi_n$} \label{sec:volume}

To compute the volume of the fixed polytope $\Pi_n^\sigma$ we will use its description as a zonotope, recalling that a zonotope can be tiled by parallelotopes as follows. If $A$ is a set of vectors, then $B \subseteq A$ is called a \emph{basis} for $A$ if $B$ is linearly independent and $\rank(B) = \rank(A)$. We define the parallelotope $\square B$ to be the Minkowski sum of the segments in $B$, that is, 
\[
\square B := \Big\{\sum_{b \in B} \lambda_b b \, : \, 0 \leq \lambda_b \leq 1 \textrm{ for each } b \in B\Big\}.
\]

\begin{theorem}\label{thm:volzonotope} \cite{D'AdderioMoci.Ehrhart, Stanleyzonotope, Ziegler}
Let $A \subset \Z^n$ be a set of lattice vectors of rank $d$ and $Z(A)$ be the associated zonotope; that is, the Minkowski sum of the vectors in $A$.
\begin{enumerate}
\item
The zonotope $Z(A)$ can be tiled using one translate of the parallelotope $\square B$ for each basis $B$ of $A$. Therefore, the  volume of the $d$-dimensional zonotope $Z(A)$ is
\[
\Vol\,(Z(A)) = \sum_{B \subseteq A \atop B \textrm{ basis}} \Vol\,(\square B).
\]
\item
For each $B \subset \Z^n$ of rank $d$, $\Vol(\square B)$ equals the index of $\Z B$ as a sublattice of $(\Span \, B) \cap \Z^n$. Using the vectors in $B$ as the columns of an $n \times d$ matrix, $\Vol(B)$ is the greatest common divisor of the minors of rank $d$.
\end{enumerate}
\end{theorem}

By Theorem \ref{thm:big}, the fixed polytope $\Pi_n^\sigma$ is a translate of the zonotope generated by the set
\[
F_\sigma =  \left\{l_k e_{\sigma_j} - l_j e_{\sigma_k} \, ; \, 1 \leq j<k \leq m\right\}.
\]
This set of vectors has a nice combinatorial structure, which will allow us to describe the bases $B$ and the volumes $\Vol\,(\square B)$ combinatorially. We do this in the next two lemmas. For a tree $T$ whose vertex set is $[m]$, let
\begin{eqnarray*}
F_T &=& \left\{l_k e_{\sigma_j} - l_j e_{\sigma_k} \, : \, j<k \textrm{ and } jk \textrm{ is an edge of } T\right\}, \\
E_T &=& \left\{\frac{e_{\sigma_j}}{l_j} - \frac{e_{\sigma_k}}{l_k} \, : \, j<k \textrm{ and }  jk \textrm{ is an edge of } T\right\}.
\end{eqnarray*}

\begin{lemma}\label{lemma:bases}
The vector configuration
\[
F_\sigma :=  \left\{l_k e_{\sigma_j} - l_j e_{\sigma_k} \, : \, 1 \leq j<k \leq m\right\}
\]
has exactly $m^{m-2}$ bases: they are the sets $F_T$ as $T$ ranges over the spanning trees on $[m]$.
\end{lemma}

\begin{proof}
The vectors in $F_{\sigma}$ are positive scalar multiples of the vectors in
\[
E_\sigma = \left\{\frac{e_{\sigma_j}}{l_j} - \frac{e_{\sigma_k}}{l_k} \, : \, 1 \leq j<k \leq m\right\}, 
\]
which are the images of the vector configuration $A_{m-1}^+ = \{f_j - f_k \, : \, 1 \leq j < k \leq m\}$ under the bijective linear map $\phi$ of \eqref{eq:phi}.
 The set $A_{m-1}^+$ is a set of  positive roots for the Lie algebra $\mathfrak{gl}_m$; its bases are known \cite{Borchardt} to correspond to the spanning trees $T$ on $[m]$, and there are $m^{m-2}$ of them by Cayley's formula \cite{Cayley}. It follows that the bases of $F_\sigma$ are precisely the sets $F_T$ as $T$ ranges over those $m^{m-2}$ trees.
\end{proof}

\begin{lemma}\label{lemma:volparallelotope}
For any tree $T$ on $[m]$ we have
\begin{eqnarray*}
1. \quad \Vol(\square F_T) &=& \prod_{i=1}^m l_i^{\deg_T(i)} \Vol(E_T),  \\
2. \quad \Vol(\square E_T) &=& \frac{\gcd(l_1, \ldots, l_m)}{l_1 \cdots l_m},
\end{eqnarray*}
where $\deg_T(i)$ is the number of edges containing vertex $i$ in $T$.
\end{lemma}

\begin{proof}
1. Since $l_k e_{\sigma_j} - l_j e_{\sigma_k} = l_jl_k  \big(\frac{e_{\sigma_j}}{l_j} - \frac{e_{\sigma_k}}{l_k}\big)$ for each edge $jk$ of $T$, and volumes scale linearly with respect to each edge length of a parallelotope, we have
\begin{eqnarray*}
\Vol(\square F_T) &=&  \Big(\prod_{jk \textrm{ edge of T}} l_jl_k \Big) \Vol(\square E_T)  \\
&=& 
\prod_{i=1}^m l_i^{\deg_T(i)} \Vol(\square E_T) 
\end{eqnarray*}
as desired.

\bigskip

2. 
The parallelotopes $\square E_T$ are the images of the parallelotopes $\square A_T$ under  the bijective linear map $\phi$ of \eqref{eq:phi}, where
\[
A_T := \{f_j - f_k \, : \, j<k , \, jk \textrm{ is an edge of } T\}.
\]
Since the vector configuration $\{f_j - f_k \, : \, 1 \leq j<k \leq m\}$ is unimodular \cite{Seymour}, all parallelotopes $\square A_T$ have unit volume. Therefore, the parallelotopes $\square E_T = \phi(\square A_T)$ have the same normalized volume, so $\Vol(E_T)$ is independent of $T$.

It follows that we can use any tree $T$ to compute $\Vol(E_T)$ or, equivalently, $\Vol(F_T)$. We choose the tree $T=\textrm{Claw}_m$ with edges $1m, 2m, \ldots, (m-1)m$. Writing the $m-1$ vectors of 
\[
F_{\textrm{Claw}_m} = \{l_me_{\sigma_i} - l_i e_{\sigma_m} \, : \, 1 \leq i \leq m-1\}
\]
as the columns of an $n \times (m-1)$ matrix, then $\Vol(F_{\textrm{Claw}_m})$ is the greatest common divisor of the non-zero maximal minors of this matrix. This quantity does not change when we remove duplicate rows; the result is the $m \times (m-1)$ matrix
\[
\begin{bmatrix}
l_m & 0 & 0 & \cdots & 0 \\
0 & l_m & 0 & \cdots & 0 \\
0 & 0 & l_m & \cdots & 0 \\
\vdots & \vdots & \vdots & \ddots & \vdots \\
0 & 0 & 0 & \cdots & l_m  \\
-l_1 & -l_2 & -l_3 & \cdots & -l_{m-1}
\end{bmatrix}.
\]
This matrix has $m$ maximal minors, whose absolute values equal $l_m^{m-2}l_1, l_m^{m-2}l_2, \ldots l_m^{m-2}l_{m-1}, l_m^{m-1}$. Therefore,
\[
\Vol(\square F_{\textrm{Claw}_m}) = l_m^{m-2} \gcd(l_1,\ldots, l_{m-1}, l_m)
\]
and part 1 then implies that
\[
\Vol(\square E_{\textrm{Claw}_m}) = \frac{\Vol(\square F_{\textrm{Claw}_m})}{l_1 \cdots l_{m-1} l_m^{m-1}}  = \frac{\gcd(l_1,\ldots, l_m)}{l_1 \cdots l_m}
\]
as desired.
\end{proof}

\bigskip

\begin{lemma} \label{lemma:trees}
For any positive integer $m \geq 2$ and unknowns $x_1, \ldots, x_m$, we have
\[
\sum_{T \textrm{ tree on } [m]} \, \prod_{i=1}^m x_i^{\deg_T(i)-1} = (x_1+\cdots+x_m)^{m-2}.
\]
\end{lemma}

\begin{proof}
We derive this from the analogous result for rooted trees \cite[Theorem 5.3.4]{EC2}, which states that
\[
\sum_{(T,r) \textrm{ rooted } \atop \textrm{ tree  on } [m]} \prod_{i=1}^m x_i^{\children_{(T,r)}(i)} = (x_1+\cdots+x_m)^{m-1} 
\]
where $\children_{(T,r)}(v)$ counts the children of $v$; that is, the neighbors of $v$ which are not on the unique path from $v$ to the root $r$.

Notice that 
\[
\children_{(T,r)}(i) = \begin{cases}
\deg_T(i)-1 & \textrm{ if } i \neq r,\\
\deg_T(i) & \textrm{ if } i = r.
\end{cases}
\]
Therefore,
\begin{eqnarray*}
\sum_{(T,r) \textrm{ rooted } \atop \textrm{ tree on } [m]} \prod_{i=1}^m x_i^{\children_{(T,r)}(i)} & = & 
\sum_{r=1}^m\Bigg(\sum_{(T,r) \textrm{ tree on } [m] \atop \textrm{ rooted at } r} x_r \prod_{i=1}^m x_i^{\deg_T(i)-1}\Bigg)   \\
& = &
\Bigg(\sum_{T \textrm{ tree on } [m]} \prod_{i=1}^m x_i^{\deg_T(i)-1}\Bigg) (x_1+\cdots+x_m)
\end{eqnarray*}
from which the desired result follows.
\end{proof}

\begin{reptheorem}{thm:main}
If $\sigma$ is a permutation of $[n]$ whose cycles have lengths $l_1, \ldots, l_m$, then the normalized volume of the slice of $\Pi_n$ fixed by $\sigma$ is 
\[
\Vol\,  \Pi_n^\sigma = n^{m-2} \gcd(l_1, \ldots, l_m).
\]
\end{reptheorem}

%

\begin{proof}
Since $\Pi_n^\sigma$ is a translate of the zonotope for the lattice vector configuration
\[
F_\sigma :=  \left\{l_k e_{\sigma_j} - l_j e_{\sigma_k} \, : \, 1 \leq j<k \leq m\right\}, 
\]
we invoke Theorem \ref{thm:volzonotope}. Using Lemmas \ref{lemma:bases}, \ref{lemma:volparallelotope}, and \ref{lemma:trees}, it follows that
\begin{eqnarray*}
\Vol \,\Pi_n^\sigma &=& \sum_{T \textrm{ tree on } [m]} \Vol(\square F_T)\\
&=& \sum_{T \textrm{ tree on } [m]} \prod_{i=1}^m l_i^{\deg_T(i)-1} \gcd(l_1, \ldots, l_m) \\
&=& (l_1+\cdots+l_m)^{m-2} \gcd(l_1, \ldots, l_m),
\end{eqnarray*}
as desired.
\end{proof}

When $\sigma = \id$ is the identity permutation, the fixed polytope is $\Pi_n^\id = \Pi_n$, and we recover Stanley's result that $\Vol \,\Pi_n = n^{n-2}$. \cite{Stanleyzonotope}

\section{Closing remarks}\label{sec:closing}

\subsection{Equivariant triangulations of the prism}

Gelfand, Kapranov, and Zelevinsky \cite{GKZ} introduced the \emph{secondary polytope}, an $(n-d-1)$-dimensional polytope $\Sigma(P)$ associated to a point configuration $P$ of $n$ points in dimension $d$. The vertices of $\Sigma(P)$ correspond to the \emph{regular triangulations} of $P$, and more generally, the faces of $\Sigma(P)$ correspond to the \emph{regular subdivisions} of $P$. Furthermore, face inclusion in $\Sigma(P)$ corresponds to refinement of subdivisions. 

The permutahedron $\Pi_n$ is the secondary polytope of the prism $\Delta_{n-1} \times I$ over the $n$-simplex. In fact, all subdivisions of the prism $\Delta_{n-1} \times I$ are regular, so the faces of the permutahedron are in order-preserving bijection with the ways of subdividing the prism $\Delta_{n-1} \times I$.

When the polytope $P$ is invariant under the action of a group $G$, Reiner \cite{Reinerequivariant} introduced the \emph{equivariant secondary polytope} $\Sigma^G(P)$, whose faces correspond to the $G$-invariant subdivisions of $P$. We call such a subdivision \emph{fine} if it cannot be further refined into a $G$-invariant subdivision.

This equivariant framework applies to our setting, since a permutation $\sigma \in S_n$ acts naturally on the prism $\Delta_{n-1} \times I$ and on the permutahedron $\Pi_n$. The following is a direct consequence of \cite[Theorem 2.10]{Reinerequivariant}. 

\begin{proposition}
The fixed polytope $\Pi_n^\sigma$ is the equivariant secondary polytope for the triangular prism $\Delta_{n-1} \times I$ under the action of $\sigma$. \end{proposition}

Thus, bearing in mind that the  faces of the $m$-permutahedron are in order-preserving bijection with the ordered set partitions of $[m]$, our Theorem \ref{thm:big} has the following consequence.

\begin{corollary}
The poset of $\sigma$-invariant subdivisions of the prism $\Delta_{n-1} \times I$ is isomorphic to the poset of ordered set partitions of $[m]$, where $m$ is the number of cycles of $\sigma$. In particular, the number of finest $\sigma$-invariant subdivisions is $m!$.
\end{corollary}

It is possible to describe the equivariant subdivisions of the prism combinatorially; we hope this will be a fun exercise for the interested reader.

\subsection{Slices of $\Pi_n$ fixed by subgroups of $S_n$}

One might ask, more generally, for the subset of $\Pi_n$ fixed by a subgroup of $H$ in $S_n$; that is, 
\[
\Pi_n^H = \{x \in \Pi_n \, : \, \sigma \cdot x = x \,\, \mathrm{ for \,\, all } \,\, \sigma \in H\}.
\]
It turns out that this more general definition leads to the same family of fixed polytopes of $\Pi_n$.

\begin{lemma}\label{lemma:fixedbyH}
For every subgroup $H$ of $S_n$ there is a permutation $\sigma$ of $S_n$ such that $\Pi_n^H = \Pi_n^\sigma$.
\end{lemma}

\begin{proof}
Let $\{\sigma_1,\dots,\sigma_r\}$ be a set of generators for $H$. Notice that a point $p \in \R^n$ is fixed by $H$ if and only if it is fixed by each one of these generators. For each generator $\sigma_t$, the cycles of $\sigma_t$ form a set partition $\pi_t$ of $[n]$. Furthermore, a point $x \in \R^n$ is fixed by $\sigma_t$ if and only if $x_j=x_k$ whenever $j$ and $k$ are in the same part of $\pi_t$. 

Let $\pi = \pi_1 \vee \cdots \vee \pi_r$ in the lattice of partitions of $[n]$; the partition $\pi$ is the finest common coarsening of $\pi_1, \ldots, \pi_r$. Then $x \in \R^n$ is fixed by each one of the generators of $H$ if and only if $x_j=x_k$ whenever $j$ and $k$ are in the same part of $\pi$. Therefore, we may choose any permutation  $\sigma$ of $[n]$ whose cycles are supported on the parts of $\pi$, and we will have $\Pi_n^H = \Pi_n^\sigma$, as desired.
\end{proof} 

\begin{example}
Consider the subset of  $\Pi_9$ fixed by the subgroup $H = \langle (173)(46)(89), (27)(68)\rangle$ of $S_9$. To be fixed by the two generators of $H$,  a point $x \in \R^9$ must satisfy
\begin{eqnarray*}
\sigma_1 = (173)(46)(89):& &  x_1=x_7=x_3, \quad x_4=x_6, \quad x_8=x_9, \\
\sigma_2 = (27)(68): & & x_2=x_7, \quad x_6=x_8,
\end{eqnarray*}
corresponding to the partitions $\pi_1 = 137|2|46|5|89$ and $\pi_2 = 1|27|3|4|5|68|9$.
Combining these conditions gives
\[
x_1=x_2=x_3=x_7,\quad  x_4=x_6=x_8=x_9,
\]
which corresponds to the join $\pi_1 \vee \pi_2 = 1237|4689|5$. For any permutation $\sigma$ whose cycles are supported on the parts of  $\pi_1 \vee \pi_2 $, such as $\sigma=(1237)(4689)$, we have $\Pi_9^H = \Pi_9^\sigma$.
\end{example}

\subsection{Lattice point enumeration and equivariant Ehrhart theory}

Theorem \ref{thm:main} is the first step towards describing the equivariant Ehrhart theory of the permutahedron, a question posed by Stapledon \cite{Stapledon}. To carry out this larger project, we need to compute the Ehrhart quasipolynomial of $\Pi_n^\sigma$, which counts the lattice points in its integer dilates:
\[
L_{\Pi_n^\sigma}(t) := \left| \, t \, \Pi_n^\sigma \cap \Z^n \right| \qquad \textrm{ for } t \in \N.
\]
New difficulties arise in this question; let us briefly illustrate some of them.

When all cycles of $\sigma$ have odd length, Theorem \ref{thm:big}.3 shows that $\Pi_n^\sigma$ is a lattice zonotope. In this case, it is not much more difficult to give a combinatorial formula for the Ehrhart polynomial, using the fact that $L_{\Pi_n^\sigma}(t)$ is an evaluation of the arithmetic Tutte polynomial of the corresponding vector configuration \cite{Ardilasurvey, D'AdderioMoci.Ehrhart}. 

In general, $\Pi_n^\sigma$ is a half-integral zonotope. Therefore, the even part of its Ehrhart quasipolynomial is also an evaluation of an arithmetic Tutte polynomial, and can be computed as above. However, the odd part of its Ehrhart quasipolynomial is more subtle. If we translate $\Pi_n^\sigma$ to become integral, we can lose and gain lattice points in the interior and on the boundary, in ways that depend on number-theoretic properties of the cycle lengths.

Some of these subtleties already arise in the simple case when $\Pi_n^\sigma$ is a segment; that is, when $\sigma$ has only two cycles of lengths $l_1$ and $l_2$. For even $t$, we simply have
\[
L_{\Pi_n^\sigma}(t) = 
\gcd(l_1,l_2)t + 1. 
\]
However, for odd $t$ we have
\[
L_{\Pi_n^\sigma}(t) = 
\begin{cases}
\gcd(l_1,l_2)t + 1 & \textrm{if $l_1$ and $l_2$ are both odd}, \\
\gcd(l_1,l_2)t & \textrm{if $l_1$ and $l_2$ have different parity}, \\
\gcd(l_1,l_2)t & \textrm{if $l_1$ and $l_2$ are both even and they have the same $2$-valuation}, \\
0 & \textrm{if $l_1$ and $l_2$ are both even and they have different $2$-valuations,}
\end{cases}
\]
where the \emph{$2$-valuation} of a positive integer is the highest power of $2$ dividing it. 

In higher dimensions, additional obstacles arise. 
Describing the equivariant Ehrhart theory of the permutahedron is the subject of an upcoming project.

\section{Acknowledgments}

Some of the results of this paper are part of the Master's theses of AS (under the supervision of FA) and ARVM (under the supervision of FA and Matthias Beck) at San Francisco State University \cite{Schindlerthesis, Vindasthesis}.  
We are grateful to Anastasia Chavez, John Guo, Andr\'es Rodr\'{\i}guez, and Nicole Yamzon for their valuable feedback during our group research meetings, and the mathematics department at SFSU for providing a wonderful environment to produce this work. We are also thankful to the referees, whose suggestions helped us improve the presentation. In particular, one of the referees pointed out the connection with equivariant secondary polytopes.
Part of this project was carried out while FA was a Simons Research Professor at the Mathematical Sciences Research Institute; he thanks the Simons Foundation and MSRI for their support.
ARVM thanks Matthias Beck and Benjamin Braun for the support and fruitful conversations. 

\small

\bibliographystyle{amsplain}
\bibliography{references}

\end{document}